\documentclass[10pt, a4paper]{article}
\usepackage[width=125mm,height=195mm]{geometry}
\usepackage{amsmath}
\usepackage{amsthm}
\usepackage{amssymb}
\usepackage[arrow, matrix, curve]{xy}
\usepackage{esint}
\usepackage{enumerate}
\usepackage{url}
\usepackage{tikz}
\usepackage{tikz-cd}
\usepackage{todonotes}

\theoremstyle{plain}
\newtheorem{thm}{Theorem}
\newtheorem{prop}[thm]{Proposition}
\newtheorem{lemma}[thm]{Lemma}
\newtheorem{cor}[thm]{Corollary}
\newtheorem*{thm*}{Theorem}

\theoremstyle{definition}
\newtheorem{defin}{Definition}

\theoremstyle{remark}
\newtheorem*{rem}{Remark}

\newcommand\blfootnote[1]{%
  \begingroup
  \renewcommand\thefootnote{}\footnote{#1}%
  \addtocounter{footnote}{-1}%
  \endgroup
}
  
\newcommand{\address}{{
\bigskip
\footnotesize
 \noindent \textsc{Faculty of Mathematics \\ University of Vienna \newline Oskar-Morgenstern-Platz 1, 1090 Vienna, Austria}\par\nopagebreak
  \noindent \textit{E-mail}: \texttt{christoph.harrach@univie.ac.at}
}}

\DeclareMathOperator{\SO}{SO}

\DeclareMathOperator{\id}{id}
\DeclareMathOperator{\GL}{GL}

\DeclareMathOperator{\im}{im}

\DeclareMathOperator{\vol}{vol}
\DeclareMathOperator{\End}{End}

\DeclareMathOperator{\gr}{gr}

\begin{document}
\title{Poisson transforms adapted to BGG-complexes}
\author{Christoph Harrach}

\newcommand{\mc}{\mathbb{C}}
\newcommand{\me}{\mathbb{E}}
\newcommand{\mh}{\mathbb{H}}
\newcommand{\ml}{\mathbb{L}}
\newcommand{\mn}{\mathbb{N}}
\newcommand{\mr}{\mathbb{R}}
\newcommand{\mv}{\mathbb{V}}
\newcommand{\mw}{\mathbb{W}}
\newcommand{\mz}{\mathbb{Z}}

\newcommand{\xa}{\mathfrak{a}}
\newcommand{\xg}{\mathfrak{g}}
\newcommand{\xh}{\mathfrak{h}}
\newcommand{\xk}{\mathfrak{k}}
\newcommand{\xm}{\mathfrak{m}}
\newcommand{\xn}{\mathfrak{n}}
\newcommand{\xp}{\mathfrak{p}}
\newcommand{\xq}{\mathfrak{q}}
\newcommand{\xu}{\mathfrak{u}}

\newcommand{\ca}{\mathcal{A}}
\newcommand{\cc}{\mathcal{C}}
\newcommand{\ce}{\mathcal{E}}
\newcommand{\ch}{\mathcal{H}}
\newcommand{\cl}{\mathcal{L}}
\newcommand{\cp}{\mathcal{P}}
\maketitle

\begin{abstract}
We present a new construction for Poisson transforms between vector bundle valued differential forms on homogeneous parabolic geometries and vector bundle valued differential forms on the corresponding Riemannian symmetric space, which can be described in terms of finite dimensional representations of reductive Lie groups. In particular, we use these operators to relate the BGG-sequences on the domain with twisted deRham sequences on the target space. Finally, we explicitly design a family of Poisson transforms between standard tractor valued differential forms for the real hyperbolic space and its boundary which are compatible with the BGG-complex.
\end{abstract}

\blfootnote{supported by project P27072-N25 of the Austrian Science Fund (FWF)}

\section{Introduction}
In the realms of harmonic analysis the Poisson transform was introduced to solve the problem of finding joint eigenfunctions for the algebra of invariant differential operators on a symmetric space \cite{helgason_GASS}. Explicitly, to any $X = G/K$ one can naturally assigned its F\"{u}rstenberg boundary $B$, which is topologically a sphere, and the natural $G$-action on $X$ extends smoothly to $B$, c.f. \cite[ch. 1.6]{borel}. Then the Poisson transforms map smooth functions on $B$ to smooth functions on $X$. Subsequently, these operators were generalized to map between sections of arbitrary homogeneous vector bundles in \cite{vanderven94} for the case of real rank $1$ and independently in \cite{olbrich95} and \cite{yang94} for arbitrary Riemannian symmetric spaces. Finally, in \cite{gaillard} P.-Y. Gaillard introduced an alternative construction of Poisson transforms for differential forms on hyperbolic spaces of even dimension which are phrased in terms of the exterior calculus. 

Although Poisson transforms are a fundamental tool in the spectral theory of Riemannian symmetric spaces, their interplay with the geometric structure of $B$ was mostly disregarded. Via the induced $G$-action the boundary $B$ can be viewed as a homogeneous space $G/P$, where $P$ is a minimal parabolic subgroup of $G$, and thus is naturally a (homogeneous) parabolic geometry. These geometric structures, which include conformal, CR and projective manifolds, have been an target of intensive study in recent years due to their r\^{o}le in the study of Poincar\'{e}-Einstein manifolds \cite{biquard_mazzeo}, AdS/CFT correspondence \cite{maldacena}, complex analysis \cite{fefferman} and representation theory \cite{kobayashi_orsted}. One of the major tools in the theory of parabolic geometries are the BGG-sequences \cite{cap_slovak_soucek, calderbank_diemer}, which in the case of homogeneous spaces are complexes. For a given tractor bundle $V$ these consist of a series of homology bundles $\ch_k = \ch_k(G/P,V)$ over $G/P$ as well as a series of differential operators $D_k \colon \Gamma(\ch_k) \to \Gamma(\ch_{k+1})$ whose cohomology computes the twisted deRham cohomology. These operators turn out to encode a great amount of information about the underlying parabolic geometry. 

The aim of this paper is to generalize Gaillard's approach as well as the scalar valued case in \cite{harrach16} to Poisson transforms between vector bundle valued differential forms on homogeneous parabolic geometries $G/P$ and their corresponding Riemannian symmetric spaces $G/K$. Explicitly, if $V \to G/P$ and $W \to G/K$ are homogeneous vector bundles, we define $G$-equivariant integral operators
\begin{align*}
 \Phi \colon \Omega^k(G/P,V) \to \Omega^{\ell}(G/K, W)
\end{align*}
whose kernels are induced by invariant elements in a finite dimensional representation of a reductive Lie group. Furthermore, since the construction of these operators is tailored to the exterior calculus on homogeneous spaces, their composition with several invariant differential operators can be expressed in terms of equivariant computations of the underlying kernels. This in turn enables us to give a simple condition on the kernel ensuring that the corresponding Poisson transform $\Phi$ naturally factors to a $G$-equivariant map
\begin{align*}
 \underline{\Phi} \colon \Gamma(\ch_k(G/P,V)) \to \Omega^{\ell}(G/K,W)
\end{align*}
which are compatible with the BGG-complex.

As an application, we consider the group $G = \SO(n+1,1)_0$, in which case $G/K$ is the real hyperbolic space of dimension $n+1$ with $G/P$ being its conformal boundary. Then for all $1 \le k \le n-1$ we construct a family of Poisson transforms $\Phi_k$ between tractor valued $k$-forms on $G/P$ and $G/K$ which naturally factor to the BGG-complex and have the following properties:
\begin{thm*}
 Let $G = \SO(n+1,1)_0$, $K$ its maximal compact subgroup, $P$ its minimal parabolic subgroup and $\mv$ the standard representation of $G$. Then for all $1 \le k \le n-1$ and any density bundle $\ce[\lambda] \to G/P$ with $\lambda \in \mr$ the $G$-equivariant operators
 \begin{align*}
  \underline{\Phi}_{k,\lambda} \colon \Gamma(\ch_k(G/P, G \times_P\mv) \otimes \ce[\lambda]) \to \Omega^k(G/K, G \times_K \mv)
 \end{align*}
induced by $\Phi_k$ satisfy the following properties:
\begin{enumerate}[(i)]
 \item The image of $\underline{\Phi}_{k,\lambda}$ is contained in the space of coclosed differential forms on $G/K$ which are eigenforms for a twisted Laplace operator with eigenvalue $-\lambda(n-2k+\lambda)$.
 \item For all $k = 1, \dotsc, n-2$ the BGG-operators $D_k$ and the covariant exterior derivative $d^{\nabla}$ induced by the tractor connection are related by
 \begin{align*}
  (k+2) d^{\nabla} \circ \underline{\Phi}_{k,0} = (-1)^{n-k}k(n-2k) \underline{\Phi}_{k+1,0} \circ D_{k}.
 \end{align*}
\end{enumerate}
\end{thm*}

This paper is organised as follows: In section \ref{sec_geometry_homogeneous} we briefly recall the definition of the BGG-complex on homogeneous parabolic geometries as well as introduce a calculus for tractor bundle valued differential forms on a Riemannian symmetric space. Afterwards, we will relate these two types of homogeneous spaces in section \ref{sec_poisson_transform} by presenting a new construction of Poisson transforms between vector valued differential forms on $G/P$ and $G/K$ and analyse their compositions with several invariant differential operators. As an application of this theory we turn in the last section to the construction of a family of Poisson transforms between tractor valued differential forms on the conformal sphere and the real hyperbolic space, which naturally descend to the BGG-complex, and study their behaviour after composition with several differential operators.

\section{Geometry of homogeneous spaces}\label{sec_geometry_homogeneous}
\subsection{Invariant connections, covariant exterior derivative}\label{sec_invariant_connections_cov_ext}
Let $G$ be a Lie group and $H$ a closed subgroup of $G$. The canonical projection $G \to G/H$ endows $G$ with the structure of a principal $H$-bundle over $G/H$. It is well known (\cite[Proposition 1.4.3]{cap_slovak}) that there is an equivalence of categories between homogeneous vector bundles $V$ over $G/H$ and finite dimensional $H$-representations $\mv$ given by the associate bundle $V = G \times_H \mv$. Furthermore, $G$-invariant sections of $V$ are in bijective correspondence with $H$-invariant elements of $\mv$, c.f. \cite[Theorem 1.4.4]{cap_slovak}.  

Recall that a $G$-invariant connection $\nabla$ on a homogeneous vector bundle $V$ gives rise to the \emph{covariant exterior derivative}, i.e. a family of $G$-invariant differential operators $d^{\nabla} \colon \Omega^k(G/H, V) \to \Omega^{k+1}(G/H, V)$ for all $0 \le k \le \dim(G/H)-1$. These operators satisfy $d^{\nabla}d^{\nabla}\omega = R^{\nabla} \wedge \omega$ for all $\omega \in \Omega^k(G/H, V)$, where $R^{\nabla}$ is the curvature of $\nabla$. In particular, the twisted deRham sequence
\[
 \begin{tikzcd}
 0 \ar{r}{} & \Gamma(V) \ar{r}{\nabla} & \Omega^1(G/H, V) \ar{r}{d^{\nabla}} & \dotso \ar{r}{d^{\nabla}} & \Omega^{\dim(G/H)}(G/H,V)
 \end{tikzcd}
\]
is a complex if and only if the connection $\nabla$ is flat. If $V$ is a homogeneous vector bundle associated to a representation of the full group $G$, then it always carries a flat $G$-invariant connection $\nabla^V$. In this case, we call $V$ a \emph{tractor bundle} and $\nabla^V$ the \emph{tractor connection}.

Let $s \in \Gamma(V)$ be a section of a homogeneous vector bundle $V = G \times_H \mv$ and $f \colon G \to \mv$ the corresponding smooth, $H$-equivariant map. The $G$-action on $f$ gives rise to an action of the universal enveloping algebra $\mathfrak{u}(\xg)$ of the Lie algebra $\xg$ of $G$ by right-invariant vector fields. In particular, the Casimir operator $\cc$ of $\xg$ induces a $G$-equivariant differential operator $D_{\cc} \colon \Gamma(V) \to \Gamma(V)$ for all homogeneous vector bundles over $G/H$. 

\begin{rem} For a real semisimple Lie group $G$ with finite centre we will consider in section \ref{sec_poisson_transform} its corresponding symmetric space $G/K$ and a parabolic geometry $G/P$ and relate geometric properties via $G$-equivariant linear operators $\Phi$ mapping between vector valued differential forms. In particular, since the operator $D_{\cc}$ is induced by the $\xg$-action on the underlying representation it will also commute with $\Phi$. Our aim in the following two sections is to express the operator $D_{\cc}$ in terms of well known differential operators on $G/K$ and $G/P$, respectively, c.f. Theorems \ref{thm_casimir_box} and \ref{thm_casimir_laplace}.
\end{rem}

\subsection{Homogeneous parabolic geometries}\label{sec_parabolic}
Let $G$ be a semisimple Lie group with Lie algebra $\xg$ and let $P \subset G$ be a parabolic subgroup with Lie algebra $\xp$. Recall that there is a corresponding $|k|$-grading on $\xg$, i.e. a vector space decomposition $\xg = \xg_{-k} \oplus \dotso \oplus \xg_k$ so that:
\begin{enumerate}[(a)]
 \item $[\xg_i, \xg_j] \subset \xg_{i+j}$, where we agree that $\xg_i = \lbrace 0 \rbrace$ for $i > |k|$,
 \item $\xg_- := \bigoplus_{i<0} \xg_i$ is generated by $\xg_{-1}$, 
 \item $\xp$ coincides with $\bigoplus_{i\ge 0} \xg_i$.
\end{enumerate}
 There is a unique element $E \in \xg_0$, called the grading element, which is characterized by $[E, X] = jX$ for all $X \in \xg_j$. The Killing form on $\xg$ induces nongenerate pairings between $\xg_i$ and $\xg_{-i}$ for all $i \ge 0$ and also a $P$-equivariant isomorphism between $(\xg/\xp)^*$ and the nilradical $\xp_+ := \bigoplus_{i > 0} \xg_i$ of $\xp$.


Exploiting the Lie algebra structure on the fibres of the cotangent bundle $T^*(G/P) \cong G \times_P \xp_+$ on $G/P$, we can use the Lie algebra homology differential to introduce a series of $G$-invariant differential operators defined on the corresponding homology bundles. These sequence of operators form a complex, called the \emph{BGG-complex}, which turns out to be a resolution of the deRham complex. Since we will rely on its construction in the sequel, we will recall it explicitly.

Consider the standard complex for Lie algebra homology of $\xp_+$ with values in a $P$-representation $\mv$. The $k$-chain space in this complex is given by the $P$-module $C_k(\xp_+, \mv) := \Lambda^k \xp_+ \otimes \mv$, and the differential is the map $\partial^* \colon C_k(\xp_+, \mv) \to C_{k-1}(\xp_+, \mv)$ defined by
\begin{align*}
 \partial^*(Z_1 \wedge \dotso \wedge Z_k \otimes v) :=\ &\sum_{i=1}^k (-1)^{i+1} Z_1 \wedge \dotso \wedge \widehat{Z_i} \wedge \dotso \wedge Z_k \otimes Z_i \cdot v \\
 &+ \sum_{i<j} (-1)^{i+j} [Z_i, Z_j] \wedge Z_1 \wedge \dotso \widehat{Z_i} \dotso \widehat{Z_j} \dotso \wedge Z_k \otimes v
\end{align*}
for all $Z_1, \dotsc, Z_k \in \xp_+$ and $v \in \mv$, where the hat denotes omission. For historical reasons this map is called the \emph{Kostant codifferential}. A direct computation shows that $\partial^*$ is $P$-equivariant and satisfies $\partial^* \circ \partial^* = 0$. In particular, its kernel and image are $P$-invariant subspaces in $C_k(\xp_+, \mv)$ and thus the $k$-th homology space $H_k(\xp_+, \mv) := \ker(\partial^*)/\im(\partial^*)$ is naturally a $P$-module, which is completely reducible.

Another property, which will be important in section \ref{sec_poisson_differential_operators}, is that the Kostant codifferential is self adjoint with respect to the wedge product. 


\begin{prop}\label{prop_Kostant_codifferential_self_adjoint}
 Let $G$ be a semisimple Lie group with Lie algebra $\xg$, $P \subset G$ a parabolic subgroup with Lie algebra $\xp$ and let $n = \dim(G/P)$. For a $P$-representation $\mv$ let $\alpha \in C_k(\xp_+, \mv)$ and $\beta \in C_{n+1-k}(\xp_+, \mv^*)$ be two chains. Then the Kostant codifferential satisfies
 \begin{align*}
  (\partial^*\alpha) \wedge \beta = (-1)^k \alpha \wedge (\partial^*\beta),
 \end{align*}
where the wedge product is formed by applying the contraction $\mv \otimes \mv^* \to \mr$.
\end{prop}

\begin{proof} For $X \in \xp_+$ denote by $\rho_X^{\mv}$ the $P$-equivariant map on $C_k(\xp_+, \mv)$ given by the tensor product of the identity on $\Lambda^k \xp_+$ and the action of $X$ on $\mv$ and let $\epsilon_X$ be the map $\phi \mapsto X \wedge \phi$. Let $\{\xi_1, \dotsc, \xi_n\}$ be a basis of $\xp_+$ and $\{\eta_1, \dotsc, \eta_n\}$ a basis of $\xg_-$ which are dual with respect to the Killing form. We interpret a $k$-chain $\phi \in C_k(\xp_+, \mv)$ as a linear map map $\Lambda^k \xg^* \to \mv$ which is trivial upon insertion of elements in $\xp$. Then a direct computation shows that $\partial^*$ can be expressed on decomposable elements as
\begin{align*}
 \partial^*\phi = \sum_{s =1}^n\left(\frac{1}{2} \sum_{t = 1}^n \epsilon_{\xi_t}\iota_{[\xi_s, \eta_t]} - \rho_{\xi_{s}}^{\mv}\right) \iota_{\eta_s}\phi
\end{align*}
and by linearity this formula holds for all $k$-chains. Subsequently, we have
\begin{align*}
(\partial^*\alpha) \wedge \beta &= \frac{1}{2} \sum_{s =1}^n\sum_{t = 1}^n \left(\epsilon_{\xi_t}\iota_{[\xi_s, \eta_t]}\iota_{\eta_s}\alpha\right) \wedge \beta - \sum_{s=1}^n \left(\rho_{\xi_{s}}^{\mv} \iota_{\eta_s}\alpha\right) \wedge \beta = (*).
\end{align*}
In order to simplify the above expression we choose the basis of $\xp_+$ to be the union of bases of the positive grading components of $\xg$. Then by duality we see that $B(\xi_s, \eta_t) \neq 0$ if and only if $s = t$, implying that we can anticommute $\epsilon_{\xi_t}$ and $\iota_{\eta_s}$. Furthermore, if the indices are the same, the bracket $[\xi_s, \eta_s]$ is contained in $\xg_0 \subset \xp$ and thus inserts trivially into $\alpha$.

Therefore, in the first sum we can move $\epsilon_{\xi_t}$ right next to $\alpha$. In the next step, we move in both sums all the interior products as well as the operator $\epsilon_{\xi_t}$ to the right hand side of the wedge product, which adds the sign $(-1)^k$ in the first sum and $(-1)^{k+1}$ in the second sum. By definition of the dual action we therefore obtain that
\begin{align*}
 (*) &= (-1)^k \alpha \wedge \left( \frac{1}{2} \sum_{s =1}^n\sum_{t = 1}^n \left(\epsilon_{\xi_t}\iota_{[\xi_s, \eta_t]}\iota_{\eta_s} \beta \right) - \sum_{s=1}^n \left( \rho_{\xi_{s}}^{\mv^*} \iota_{\eta_s}\beta\right)\right) = (-1)^k \alpha \wedge (\partial^* \beta)
\end{align*}
as claimed.
\end{proof}

Let $V \to G/P$ be the vector bundle associated to the $P$-representation $\mv$. Since the subalgebra $\xp_+$ is isomorphic to $(\xg/\xp)^*$ as $P$-modules, the Kostant codifferential induces a $G$-invariant bundle map $\Lambda^k T^*(G/P) \otimes V \to \Lambda^{k-1} T^*(G/P) \otimes V$ for all $1 \le k \le \dim(G/P)$ and also a tensorial map on the level of sections, which we both denote by the same symbol. Furthermore, the kernel and image of this bundle map are $G$-invariant subbundles in $\Lambda^k T^*(G/P) \otimes V$ with underlying representations $\ker(\partial^*)$ and $\im(\partial^*)$, and we denote these subbundles by the same symbols. By naturality, the corresponding homology bundle $\ch_k(G/P, V)$ is associated to the $P$-module $H_k(\xp_+, \mv)$ and we denote by $\pi$ the canonical projection $\ker(\partial^*) \to \ch_k(G/P,V)$ as well as the induced map on the level of sections.

In the case of a tractor bundle $V$ we can use this setting to define a series of natural differential operators mapping between sections of homology bundles. In order to do so, let $d^V$ be the covariant exterior derivative induced by the tractor connection. Then it was shown in \cite{cap_slovak_soucek} (and simplified in \cite{calderbank_diemer}) that there exists a natural differential operator $L \colon \Gamma(\ch_k(G/P,V)) \to \Gamma(\ker(\partial^*))$, called the \emph{splitting operator}, which is determined by $\pi \circ L = \id$ and $\partial^* \circ d^V \circ L = 0$. Subsequently, we can define for all $0 \le k \le \dim(G/P)-1$ the \emph{($k$-th) BGG-operator} $D^V_k \colon \ch_k(G/P,V) \to \ch_{k+1}(G/P,V)$ via $D^V_k := \pi \circ d^V \circ L$. Since the tractor connection is flat, these operators satisfy $D_{k+1}^V \circ D_k^V = 0$, and the cohomology of the induced \emph{BGG-complex} 
\[ 
 \begin{tikzcd}
  0 \ar{r} &  \Gamma(\ch_0(G/P, V)) \ar{r}{D^{V}} &   \dotsb \ar{r}{D^{V}} & \Gamma(\ch_{\dim(G/P)}(G/P, V)) \ar{r} & 0
\end{tikzcd}
\]
computes the twisted deRham cohomology, c.f. \cite[Corollary 4.15]{cap_slovak_soucek}.

In order to relate the BGG-complex with Poisson transforms in the sequel, we define the \emph{(curved) box operator} $\Box^R := d^V\partial^* + \partial^*d^V$. Then by \cite[Proposition 5.5]{calderbank_diemer} each homology class in $\Gamma(\ch_k(G/P,V))$ has a unique representative in $\Gamma(\ker(\Box^R))$ which coincides with the image of $L$. 

\begin{thm}\label{thm_casimir_box}
 Let $G$ be a semisimple Lie group with Lie algebra $\xg$, $P$ a parabolic subgroup of $G$ and $V$ a tractor bundle over $G/P$ with underlying representation $\mv$. Let $D_{\cc} \colon \Omega^*(G/P,V) \to \Omega^*(G/P,V)$ be the differential operator induced by the Casimir element $\cc$ of $\xg$ and denote by $\tau_{\cc}^V$ the tensorial operator on $\Omega^*(G/P,V)$ induced by the tensor product of the identity on $\Lambda^k \xp_+$ and the action of $\cc$ on $\mv$. Then the curved box operator $\Box^R$ satisfies
 \begin{align*}
  D_{\cc}(\alpha) = 2 \Box^R(\alpha) + \tau^V_{\cc}(\alpha).
 \end{align*}
 for all $\alpha \in \Omega^k(G/P,V)$.
\end{thm}
\begin{proof}
 \cite[Corollary 1]{cap_soucek_casimir}.
\end{proof}

\subsection{The Casimir operator on symmetric spaces}\label{sec_symmetric_space}
Let $G$ be a semisimple Lie group with finite centre and Lie algebra $\xg$, let $\theta \colon \xg \to \xg$ be a Cartan involution with corresponding Cartan decomposition $\xg = \xk \oplus \xq$ and $K \subset G$ be the maximal compact subgroup with Lie algebra $\xk$. The Lie bracket on $\xg$ satisfies $[\xk, \xk] \subset \xk$, $[\xk, \xq] \subset \xq$ and $[\xq,\xq] \subset \xk$, showing that $\xq$ is naturally a $\xk$-module as well as a $K$-module. Therefore, we obtain an isomorphism between $\xg/\xk$ and $\xq$ as $K$-modules and thus an identification of the tangent bundle $T(G/K)$ with the associated bundle $G \times_K \xq$. In particular, the restriction of the Killing form to $\xq$ is positive definite and $K$-invariant and hence induces a $G$-invariant Riemannian metric on $G/K$. Furthermore, the Cartan decomposition also induces a distribution $Q$ of $TG$ which is horizontal for the projection $G \to G/K$. This in turn induces a $G$-invariant connection on each homogeneous vector bundles over $G/K$, called the \emph{canonical connection}.

If $V$ is a tractor bundle with underlying representation $\rho \colon G \to \GL(\mv)$, then the tractor connection $\nabla^V$ is also a $G$-invariant connection on $V$. Furthermore, by \cite[Proposition 3.3.1]{cap_slovak} there is a $K$-invariant inner product $\langle \ , \ \rangle_{\mv}$ on $\mv$ characterized by $\langle X \cdot v, w \rangle_{\mv} = -\langle v, \theta(X)\cdot w \rangle_{\mv}$ for all $v$, $w \in \mv$ and $X \in \xg$, and we denote by $h$ the induced $G$-invariant bundle metric on $V$. Then we can define an additional $G$-invariant connection $\nabla^{\theta}$ on $V$, called the \emph{twisted tractor connection}, as the adjoint connection to $\nabla^V$ with respect to $h$, i.e. determined by the relation
\begin{align*}
 h(s, \nabla^{\theta}_{\xi} t) := \xi \cdot h(s, t) - h(\nabla^V_\xi s, t)
\end{align*}
for all $s$, $t \in \Gamma(V)$ and $\xi \in \mathfrak{X}(G/K)$. By adjointness it immediately follows that the connection $\nabla^{\theta}$ is also flat.

For the rest of this section we construct a calculus on tractor bundle valued differential forms on $G/K$. First, using the Euclidean bundle metric on $V$ we define for all $V$-valued $k$-forms $\alpha$ and $V$-valued $\ell$-forms $\beta$ a differential form $\alpha \wedge_h \beta \in \Omega^{k+\ell}(G/K)$ by the usual wedge product on differential forms composed with the contraction $h \colon \Gamma(V) \times \Gamma(V) \to C^{\infty}(G/K)$. Second, the bundle metric $h$ on $V$ together with the Riemannian metric on $G/K$ induce a pointwise inner product $\langle \ , \ \rangle$ on $\Omega^k(G/K, V)$ for all $k = 0, \dotsc, \dim(G/K)$ as well as an $L^2$-inner product. The former can be used to define a Hodge star operator 
\begin{align*}
 \ast &\colon \Omega^k(G/K,V) \to \Omega^{\dim(G/K)-k}(G/K, V), & \alpha \wedge_h \ast \beta &:= \langle \alpha, \beta \rangle \vol
\end{align*}
for all $\alpha$, $\beta \in \Omega^k(G/K, V)$, where $\vol$ is the volume form on $G/K$. Next, the tractor connection $\nabla^V$ and the twisted tractor connection $\nabla^{\theta}$ are both flat, implying that their covariant exterior derivatives $d^V$ and $d^{\theta}$ both induce a twisted deRham complex. Subsequently, we define the \emph{covariant codifferential} $\delta^V \colon \Omega^k(G/K, V) \to \Omega^{k-1}(G/K,V)$ to be the formal adjoint to $d^V$ with respect to the $L^2$-inner product, which can also be described by $\delta^V = (-1)^k \ast^{-1} d^{\theta} \ast$. Finally, we define the \emph{covariant Laplace} by $\Delta^V := d^V \delta^V + \delta^V d^V$, which by construction is formally selfadjoint with respect to the $L^2$-inner product.

\begin{thm}\label{thm_casimir_laplace}
 Let $G$ be a semisimple Lie group with finite centre and Lie algebra $\xg$ and $K \subset G$ a maximal compact subgroup. Let $\mv$ be a $G$-representation with corresponding tractor bundle $V$ over $G/K$. Let $D_{\cc} \colon \Omega^*(G/K, V) \to \Omega^*(G/K,V)$ be the differential operator induced by the Casimir operator $\cc$ on $\xg$ and $\tau_{\cc}^V$ the tensorial map on $\Omega^*(G/K,V)$ which is induced by the action of $\cc$ on $\mv$ tensorized with the identity on $\Lambda^k \xq^*$. Then for all $\alpha \in \Omega^k(G/P,V)$ we have
 \begin{align*}
  D_{\cc}\alpha = -\Delta^V\alpha + \tau_{\cc}^V(\alpha).
 \end{align*}
 \end{thm}

\begin{proof} \cite[p.385]{matsushima_murakami63} in the case $\Gamma = \{e\}$.
 \end{proof}

\begin{rem}
 Note that in the case of scalar valued differential forms Theorem \ref{thm_casimir_laplace} degenerates to the well-known fact that $D_{\cc}$ is the negative of the Laplace-Beltrami operator.
\end{rem}

\section{Poisson transforms and their properties}\label{sec_poisson_transform}

\subsection{Equivariant operators and the BGG-complex}\label{sec_equivariant_operators}
Throughout this section let $G$ be a semisimple Lie group with finite centre, $K$ a maximal compact subgroup and $P$ a parabolic subgroup of $G$. For homogeneous vector bundles $V \to G/P$ and $W \to G/K$ we consider $G$-equivariant continuous linear operators 
\begin{align*}
 \Phi \colon \Omega^k(G/P, V) \to \Omega^{\ell}(G/K,W)
\end{align*}
which we will use to relate the geometries of $G/K$ and $G/P$. In particular, if $V$ is a tractor bundle we focus on equivariant operators $\Phi$ as above which naturally descend to the BGG-complex.

Recall that the image of the splitting operator $L \colon \Gamma(\ch_k(G/P,V)) \to \Gamma(\ker(\partial^*))$ is the unique representative of a homology class contained in the kernel of the curved box operator $\Box^R$. Therefore, we could assume that $\Phi$ is trivial on the kernel of $\Box^R$, as this ensures to obtain an induced operator on sections of the homology bundles. However, this forces us to compute the splitting operator explicitly, which is a tedious task in general. Thus, we will instead consider an operator $\Phi$ which is trivial on the image of the Kostant codifferential and therefore factors to a $G$-equivariant linear operator
\begin{align*}
 \underline{\Phi} \colon \Gamma(\ch_k(G/P,V)) \to \Omega^{\ell}(G/K,W).
\end{align*}
Furthermore, we also assume that the composition $\Phi \circ d^V \circ \partial^*$ is trivial, as this enables us to compute the image of $D^V\sigma$ for $\sigma \in \Gamma(\ch_k(G/P,V))$ without the use of the splitting operator. Indeed, for any $\alpha \in \Gamma(\ker(\partial^*))$ representing $\sigma$ we can find a $V$-valued $(k-1)$-form $\beta$ on $G/P$ such that $L(\sigma) = \alpha + \partial^*\beta$. By definition of the BGG-operator we obtain that
\begin{align*}
 D^V\sigma = \pi(d^VL(\sigma)) = \pi(d^V\alpha) + \pi(d^V\partial^*\beta),
\end{align*}
where $\pi \colon \Gamma(\ker(\partial^*)) \to \Gamma(\ch_k(G/P,V))$ is the canonical projection, and since $\underline{\Phi} \circ \pi$ coincides with $\Phi$ this implies that $\underline{\Phi}(D^V\sigma) = \Phi(d^V\alpha)$. 

The following Theorem shows that the two ways of descending the operator $\Phi$ to the BGG-complex are equivalent as well as the geometric significance for the corresponding symmetric space.

\begin{thm}\label{thm_harmonic_bgg}
 Let $G$ be a semisimple Lie group with finite centre and Lie algebra $\xg$, $K \subset G$ a maximal compact subgroup and $P \subset G$ a parabolic subgroup. Let $\mv$ be a $G$-representation and let $V_K$ and $V_P$ the corresponding tractor bundles over $G/K$ and $G/P$, respectively. Let $\Phi \colon \Omega^k(G/P, V_P) \to \Omega^{\ell}(G/K,V_K)$ be a $G$-equivariant linear operator which is also equivariant with respect to the induced $\xg$-actions. Then the following are equivalent:
 \begin{enumerate}[(i)]
  \item The operator $\Phi$ satisfies $\Phi \circ \partial^* = 0$ and $\Phi \circ d^{V_P} \circ \partial^* = 0$,
  \item The operator $\Phi$ is trivial on the image of $\Box^R$,
  \item The image of $\Phi$ is contained in the kernel of $\Delta^{V_K}$.
  \end{enumerate}
\end{thm}
\begin{proof}
 (i) $\Leftrightarrow$ (ii): By definition of $\Box^R$ the implication $(i) \Rightarrow (ii)$ is trivial. Conversely, by \cite[Theorem 5.2]{calderbank_diemer} the box operator is invertible on the image of $\partial^*$, implying
 \begin{align*}
  \Phi \circ \partial^* = \Phi \circ \Box^R \circ (\Box^R)^{-1} \circ \partial^* = 0,
 \end{align*}
and hence also $\Phi \circ d^{V_P} \circ \partial^* = 0$.

(ii) $\Leftrightarrow$ (iii): Let $\tau_{\cc}^{V_K}$ and $\tau_{\cc}^{V_P}$ denote the $G$-equivariant tensorial maps on $\Omega^\ell(G/K,V_K)$ and $\Omega^{k}(G/P,V_P)$, respectively, which are induced by the action of the Casimir element $\cc$ of $\xg$ on $\mv$. Since the operator $\Phi$ is $\xg$-equivariant it commutes with the differential operator $D_{\cc}$, so by combining Theorem \ref{thm_casimir_box} and Theorem \ref{thm_casimir_laplace} we obtain that
\begin{align*}
 2 \Phi(\Box^R\alpha) + \Phi(\tau_{\cc}^{V_P}\alpha) = \Phi(D_{\cc}\alpha) = D_{\cc}\Phi(\alpha) = -\Delta^{V_K}\Phi(\alpha) + \tau_{\cc}^{V_K}(\Phi(\alpha))
\end{align*}
for all $\alpha \in \Omega^k(G/P, V)$. Finally, we have $\tau_{\cc}^{V_K} \circ \Phi = \Phi \circ \tau_{\cc}^{V_P}$ by $\xg$-equivariance of $\Phi$, implying that $\Delta^{V_K}\Phi(\alpha) = 0$ if and only if $\Phi(\Box^R\alpha) = 0$.
\end{proof}

\subsection{Definition of Poisson transforms}\label{sec_definition_poisson_transform}
Following the previous section, we now present a way of constructing linear $G$-equivariant integral operators between vector bundle valued differential forms on $G/P$ and $G/K$. Consider the product manifold $\cp := G/K \times G/P$, which is endowed with a canonical $G$-action from the left, and denote the canonical projection from $\cp$ onto $G/K$ and $G/P$ by $\pi_K$ and $\pi_P$, respectively. The product structure induces a pointwise decomposition
\begin{align*}
 \Lambda^k T^*\cp &= \bigoplus_{p+q=k} \Lambda^{p,q} T^*\cp \cong \bigoplus_{p+q=k} \Lambda^p T^*(G/K)\otimes \Lambda^q T^*(G/P),
\end{align*}
the latter isomorphism being induced by the canonical projections. We denote by $\Omega^{p,q}(\cp)$ the space of sections of $\Lambda^{p,q}T^*\cp$ and say that its elements are of \emph{bidegree $(p,q)$} (or shortly, \emph{$(p,q)$-forms}). Note that by construction the wedge product is compatible with the bigrading and that forms with different bidegree are linearly independent. Similarly, if $E$ is a vector bundle over $\cp$ we can define the space $\Omega^{p,q}(\cp, E)$ of $E$-valued differential forms on $\cp$ of bidegree $(p,q)$ in the same way. 

Let $\mv$ be a finite dimensional $P$-representation, $\mw$ a finite dimensional $K$-representation and denote their associated homogeneous vector bundles by $V := G \times_K \mv$ and $W := G \times_P \mw$, respectively. Then the tensor bundle $E := \pi_K^*W \otimes \pi_{P}^*V^*$ is naturally a vector bundle over $\cp$. If $n$ denotes the dimension of the compact manifold $G/P$, we define for all $\phi \in \Omega^{\ell, n-k}(\cp, E)$ a linear operator
\begin{align*}
 \Phi &\colon \Omega^k(G/P, V) \to \Omega^{\ell}(G/K, W), & \alpha &\mapsto \fint_{G/P} \phi \wedge \pi_{P}^*\alpha.
\end{align*}
Here the wedge product is formed via the canonical pairing between the vector bundles $E$ and $\pi_{P}^*V$. Furthermore, it is easy to see that the operator $\Phi$ is $G$-equivariant if and only if its corresponding kernel $\phi$ is a $G$-invariant differential form.

\begin{defin}\label{def_poisson_transform}
 If the integral operator $\Phi$ is $G$-equivariant we call $\Phi$ a \emph{Poisson transform} and $\phi$ its corresponding \emph{Poisson kernel}.
\end{defin}
%
%

Thus, in order to produce Poisson transforms we need to find $G$-invariant $E$-valued differential forms on $\cp$, which can be reduced to computations in finite dimensional representations. Indeed, the maximal compact subgroup $K$ acts transitively on $G/P$, implying that the group $G$ acts transitively on the product $\cp$. In particular, if we denote by $M$ the intersection of $K$ and $P$ we see that $\cp$ is isomorphic to $G/M$. Furthermore, denoting the induced projections from $G/M$ to $G/P$ and $G/K$ by the same symbols, a moment of thought shows that the vector bundle $E = \pi_K^*W \otimes \pi_P^*V^*$ over $G/M$ is the homogeneous bundle associated to the $M$-representation $\me := \mw \otimes \mv^*$. All in all, we see that Poisson kernels can be viewed as $G$-invariant elements in $\Omega^*(G/M, E)$, which in turn correspond to $M$-invariant elements in the underlying representation $\Lambda^* (\xg/\xm)^* \otimes \me$.

\begin{thm}\label{thm_bijective_correspondence_transforms_kernel}
Let $\xg$ and $\xm$ denote the Lie algebras of $G$ and $M$, respectively. Then there is a bijective correspondence
\begin{align*}
 \left\lbrace \substack{\text{Poisson transforms} \\ \Phi \colon \Omega^k(G/P, V) \to \Omega^\ell(G/K, W)}\right\rbrace \Leftrightarrow \left\lbrace \substack{M-\text{invariant elements} \\\text{in } \Lambda^k (\xg/\xm)^* \otimes \me}\right\rbrace.
\end{align*}
\end{thm}

\begin{rem}
The pullback of a section of a line bundle $L$ over $G/P$ along $\pi_P$ is a smooth function on $G/M$. Therefore, each Poisson kernel $\phi \in \Omega^{\ell, n-k}(\cp, E)$ induces a family of $G$-equivariant operators $\Omega^k(G/P, V\otimes L) \to \Omega^{\ell}(G/K, W)$ by the same formula.
\end{rem}


\subsection{Compatibility with differential operators}\label{sec_poisson_differential_operators}
In this section we will prove compatibility results between Poisson transforms and several differential operators acting on vector valued differential forms on $G/P$ and $G/K$, respectively, which can be expressed in terms of the corresponding Poisson kernel. Since we are mostly interested Poisson transforms which descend to the BGG-complex we will mainly focus on the case of tractor bundle valued differential forms in the sequel. 

Let $\mv$ and $\mw$ be $G$-representations with associated tractor bundles $V \to G/P$ and $W \to G/K$. The pullbacks $V_M$ of $V$ along $\pi_P$ and $W_M$ of $W$ along $\pi_K$ are both tractor bundles over $G/M$ with the same underlying vector spaces, regarded as $M$-representations. In particular, the tensor product $E = W_M \otimes V_M^*$ is also a tractor bundle over $G/M$ with underlying $G$-representation $\mw \otimes \mv^*$.

Recall from section \ref{sec_parabolic} that the Kostant codifferential $\partial^*$ on $\Omega^*(G/P,V^*)$ is a $G$-equivariant tensorial operator, so it can be lifted to a $G$-equivariant operator $\partial_P^*$ on $\Omega^{0,*}(G/M, V_M^*)$ via the projection $\pi_P$. Moreover, for $\alpha \in \Omega^{p,0}(G/M, W_M)$ and $\beta \in \Omega^{0,q}(G/M, V_M^*)$ we define $\partial_P^*$ on their wedge product by $\partial_P^*(\alpha \wedge \beta) := (-1)^p \alpha \wedge (\partial_P^*\beta)$. By linear extension this defines a $G$-equivariant tensorial operator 
\begin{align*}
 \partial_P^* \colon \Omega^{p,q}(G/M, E) \to \Omega^{p,q-1}(G/M, E),
\end{align*}
which we call the \emph{$P$-codifferential}.

Similarly, the Hodge star operator $\ast$ on $W$-valued differential forms on $G/K$ is $G$-equivariant and tensorial and thus induces an operator $\ast_K$ on $\Omega^{p,0}(G/M, W_M)$. Furthermore, for $\alpha \in \Omega^{p,0}(G/M, W_M)$ and $\beta \in \Omega^{0,q}(G/M, V_M^*)$ we define $\ast_K(\alpha \wedge \beta) := (\ast_K\alpha) \wedge \beta$ and extend this linearly to a $G$-equivariant operator
\begin{align*}
 \ast_K \colon \Omega^{p,q}(G/M, E) \to \Omega^{\dim(G/K)-p, q}(G/M, E),
\end{align*}
which we call the \emph{$K$-Hodge star}.

Next, the pullbacks of the tractor connections on $V^*$ and $W$ are $G$-invariant flat connections on $V_M^*$ and $W_M$. Thus, their tensor product is $G$-invariant connection $\nabla^E$ on the tractor bundle $E$ over $G/M$, which by naturality coincides with the tractor connection. The induced covariant exterior derivative $d^{E}$ can be decomposed into partial derivatives $d^{E} = d_K + d_P$, where the first and second operator map differential forms of bidegree $(p,q)$ to forms of bidegree $(p+1,q)$ and $(p, q+1)$, respectively. The partial derivatives will again be $G$-equivariant differential operators by construction, and we call $d_K$ the \emph{$K$-derivative} or \emph{$K$-differential} and similarly for $d_P$. 

Furthermore, recall from section \ref{sec_symmetric_space} that $W$ is also endowed with the twisted tractor connection $\nabla^{\theta}$, which is also $G$-invariant by construction. Therefore, we obtain an additional $G$-invariant connection $\nabla^{E,\theta}$ on $E$ and a decomposition of the induced covariant exterior derivative $d^{E,\theta} = d_K^{\theta} + d_P^{\theta}$ into partial derivatives. In regards of the definition of the twisted tractor connection it follows that the restrictions of $\nabla^E$ and $\nabla^{E,\theta}$ to the subbundle $\ker(T\pi_K) = G \times_M (\xk/\xm)$ of $T(G/M)$ coincide, which in turn implies that the partial derivatives $d_P$ and $d_P^{\theta}$ coincide.

Combining $d_K^{\theta}$ with the $K$-Hodge star we define the \emph{$K$-codifferential} $\delta_K$ on $E$-valued $(p,q)$-forms on $G/M$ by $\delta_K := (-1)^p \ast_K^{-1} d^{\theta}_K \ast_K$, and subsequently the \emph{$K$-Laplace} by $\Delta_K := d_K\delta_K + \delta_K d_K$, which are both $G$-invariant differential operators by construction.

All the differential operators on $\Omega^*(G/M, E)$ defined above are $G$-equivariant and thus have $M$-equivariant counterparts on the level of the underlying representations, which we will denote by the same symbols.

\begin{thm}\label{thm_Poisson_transform_differential_operators}
 Let $G$ be a semisimple Lie group with finite centre, $K$ its maximal compact subgroup, $P \subset G$ a parabolic subgroup and $V \to G/P$ and $W \to G/K$ be tractor bundles. Let $\Phi \colon \Omega^k(G/P,V) \to \Omega^\ell(G/K,W)$ be a Poisson transform with underlying Poisson kernel $\phi$ of bidegree $(\ell,n-k)$.
 \begin{enumerate}[(i)]
  \item The compositions $\Phi \circ d^V$ and $\Phi \circ \partial^*$ are Poisson transforms with kernels $(-1)^{n-k+\ell+1} d_P\phi$ and $(-1)^{n-k+\ell} \partial_P^*\phi$, respectively.
  \item The compositions $d^W \circ \Phi$, $\ast \circ \Phi$, $\delta^W \circ \Phi$ and $\Delta^W \circ \Phi$ are Poisson transforms with kernels $d_K\phi$, $\ast_K\phi$, $\delta_K\phi$ and $\Delta_K\phi$, respectively.
 \end{enumerate}
 Moreover, we have $\ast_K d_P = (-1)^{n+1} d_P\ast_K$ as well as $\ast_K \partial_P^* = (-1)^{n+1} \partial_P^* \ast_K$. In particular, both the $P$-differential and $P$-codifferential anticommute with $d_K$, $\delta_K$ and $\Delta_K$.
\end{thm}

\begin{proof} Let $V_M$ be the pullback bundle of $V$ along $\pi_P$, let $W_M$ the pullback of $W$ along $\pi_K$ and put $E := W_M \otimes V_M^*$. By definition of the pullback connection, the corresponding covariant exterior derivatives satisfy the relations $\pi_P^* \circ d^V = d^{V_M^*} \circ \pi_P^*$.  Thus, we obtain for all $\alpha \in \Omega^*(G/P,V)$ that 
\begin{align*}
 d^{W_M}(\phi \wedge \pi_P^*\alpha) + (-1)^{n-k+\ell+1} \phi \wedge \pi_P^*d^V\alpha = (d^E\phi) \wedge \pi_P^*\alpha
\end{align*}
Furthermore, using the local expression of the fibre integral in a bundle chart for $G/M \to G/K$ it can be shown that $d^{W} \circ \fint_{G/P} = \fint_{G/P} \circ d^{W_M}$ via an easy generalization of \cite[VII, Proposition X]{greub_halperin_vanstone72}. Therefore, splitting $d_E\phi$ into partial derivatives and integrating both sides of the above equation we deduce
\begin{align}\label{eqn_covariant_derivative_proof}
 d^{W}\Phi(\alpha) + (-1)^{n-k+\ell+1} \Phi(d^V\alpha) = \fint_{G/P} (d_K\phi) \wedge \pi_P^*\alpha + \fint_{G/P} (d_P\phi) \wedge \pi_P^*\alpha.
\end{align}
\begin{enumerate}[(i)]
\item If $\alpha$ is of degree $k-1$ the first summand on the left hand side of (\ref{eqn_covariant_derivative_proof}) is trivial. Moreover, on the right hand side the integrand involving the $K$-derivative is of total bidegree $(\ell+1,n-1)$ and thus vanishes after integration over $G/P$. For the second formula recall that the Kostant codifferential is self-adjoint due to Proposition \ref{prop_Kostant_codifferential_self_adjoint}. Therefore, by definition of the $P$-codifferential we get 
\begin{align*}
 \left(\partial_P^*\phi \right) \wedge \pi_P^*\alpha = (-1)^{n-k+\ell} \phi \wedge \pi_P^*\partial^*\alpha,
\end{align*}
and the claimed formula follows by integrating both sides.

\item If $\alpha$ is of degree $k$ the second summand on the left hand side of equation (\ref{eqn_covariant_derivative_proof}) is trivial, whereas on the right hand side the integrand involving the $P$-derivative is of bidegree $(\ell, n+1)$ and thus vanishes after integration. Next, the Hodge star operator is tensorial, so by the local description of the fibre integral we immediately conclude that $\ast \circ \Phi$ is the Poisson transform associated to the kernel $\ast_K\phi$. Combining these two results the rest follows immediately.
\end{enumerate}
The commutation relations between the differential operators on $G/M$ follow from applying the corresponding operators on $G/K$ and $G/P$ simultaneously to $\Phi$. As an example, if we first apply (i) and then (ii) to the composition $\ast \circ \Phi \circ d^V$, its kernel is given by $(-1)^{n-k+\ell+1}\ast_K d_P\phi$, whereas the other way round it computes as $(-1)^{k+\ell} d_P \ast_K\phi$. 
\end{proof}

\begin{rem}
 After exchanging the differential operators for vector valued differential forms by their analogues on scalar valued forms, similar compatibility results for Poisson transforms between scalar valued differential forms hold, compare with \cite[Proposition 3]{harrach16}.
\end{rem}


As a Corollary we obtain an easy condition when a Poisson transform factors to the BGG-complex on $G/P$, c.f section \ref{sec_equivariant_operators}, which can be phrased entirely in terms of its Poisson kernel.
\begin{cor}\label{cor_bgg_poisson_kernel}
 Let $G$ be a semisimple Lie group with finite centre, $K$ a maximal compact subgroup of $G$, $P \subset G$ a parabolic subgroup and let $V \to G/P$ and $W \to G/K$ be tractor bundles. Then a Poisson transform $\Phi \colon \Omega^k(G/P,V) \to \Omega^\ell(G/K,W)$ factors to the BGG-complex on $G/P$ if and only if its underlying kernel $\phi$ satisfies $\partial_P^*\phi = 0$ and $\partial_P^*d_P\phi = 0$.
\end{cor}
\begin{proof}
 Follows immediately from Theorem \ref{thm_Poisson_transform_differential_operators}.
\end{proof}

\begin{rem}
If we assume in addition that $V$ and $W$ have the same underlying $G$-representation, then by Theorem \ref{thm_harmonic_bgg} we know that the images of Poisson transforms which factor to the BGG-complex are automatically harmonic. Thus, Corollary \ref{cor_bgg_poisson_kernel} provides an efficient way to design Poisson transforms with harmonic images via computations of kernels of linear maps between finite dimensional vector spaces. 
\end{rem}

\section{Poisson transforms for real hyperbolic space}
Let $G = \SO(n+1,1)_0$ be the identity component of the special orthogonal group of signature $(n+1,1)$ with maximal compact subgroup $K$ and minimal parabolic subgroup $P$. Then $G/K$ is the real hyperbolic space of dimension $n+1$, whereas $G/P$ is the conformal $n$-sphere. Let $\mv = \mr^{n+1,1}$ be the standard representation of $G$ and denote the corresponding tractor bundles over $G/K$ and $G/P$ by $V_K$ and $V_P$, respectively, which are called the \emph{standard tractor bundles}. For the rest of this article we explicitly construct Poisson transforms
\begin{align*}
 \Phi \colon \Omega^k(G/P, V_P) \to \Omega^\ell(G/K,V_K)
\end{align*}
which factor to the BGG-complex, and determine their properties regarding differential operators on $G/K$ and $G/P$. 

\subsection{Geometric preliminaries}\label{sec_geometric_preliminaries}
Consider the bilinear form $b$ on $\mv = \mr^{n+2}$ of signature $(n+1,1)$ corresponding to the symmetric matrix 
\begin{align*}
 S = \begin{pmatrix} 0 & 0 & 1 \\ 0 & \id_n & 0 \\ 1 & 0 & 0 \end{pmatrix}
\end{align*}
and realize $\SO(n+1,1)$ as the set of all $g \in GL(n+2)$ which satisfy $g^tSg = S$ and $\det(g) = 1$. Writing elements of $G$ as block matrices with the same block sizes as the matrix $S$ the parabolic subgroup $P$ is given by
\begin{align*}
 P = \left\lbrace \begin{pmatrix} a & -aY^tB & -\frac{a}{2}|Y|^2 \\ 0 & B & Y\\ 0 & 0 & a^{-1} \end{pmatrix} \middle| a \in \mr^*, Y \in \mr^n, B \in \SO(n) \right\rbrace,
\end{align*}
whereas the maximal compact subgroup is the stabilizier of the global Cartan involution $g \mapsto (g^T)^{-1}.$ In particular, the subgroup $M$ consists of all matrices in $P$ with $a = 1$ and $Y = 0$. Furthermore, we define the group $A$ as the set of all matrices in $P$ with $B = \id$ and $Y = 0$ as well as the group $N$ as all matrices in $P$ with $B = \id$ and $a = 1$. In this way, the Langlands decomposition of $P$ is given by $P = MAN$. 

Deriving the defining equation for the matrix representation of $G$ the Lie algebra $\xg = \mathfrak{so}(n+1,1)$ is given by
\begin{align*}
 \xg = \left\lbrace \begin{pmatrix} a & -Y^t & 0 \\ X & B & Y \\ 0 & -X^t & -a \end{pmatrix} \middle| a \in \mr, X, Y \in \mr^n, B \in \mathfrak{so}(n)\right\rbrace.
\end{align*}
The block form of elements in $\xg$ induces a $|1|$-grading $\xg_{-1} \oplus \xg_0 \oplus \xg_1$ of $\xg$. Moreover, the Lie algebra $\xm$ of $M$ consists of those matrices with $B \in \mathfrak{so}(n)$ being its only nontrivial entry and thus is contained in $\xg_0$. In particular, the $|1|$-grading of $\xg$ is invariant under the natural $\xm$-action and thus induces a decomposition of $\xg/\xm$ into the direct sum 
\begin{align*}
 \xg/\xm = (\xg/\xm)_{-1} \oplus (\xg/\xm)_{0} \oplus (\xg/\xm)_{-1}.
\end{align*}
We identify each element $\xi \in \xg/\xm$ with its representative in $\xg$ whose $\mathfrak{so}(n)$-part is trivial, and we write $\xi = (X, a, Y)$ for the corresponding matrix. In this picture, the action of $m \in M$ corresponding to $B \in \SO(n)$ is given by $m \cdot \xi = (BX, a, BY)$. Furthermore, the subspace $\xp/\xm$ is generated by $E := (0,1,0)$ and the elements $F_X := (0, 0, X)$ for $X \in \mr^n$, whereas $\xk/\xm$ is given by all elements of the form $G_Y := (Y, 0, Y)$ for $Y \in \mr^n$.

\begin{lemma}\label{lem_E_invariant}
 The element $E^* \in (\xg/\xm)^*$ dual to $E$ is $M$-invariant and of bidegree $(1,0)$. Furthermore, its $K$-derivative $d_KE^*$ is trivial, whereas its $P$-derivative $d_PE^*$ is nondegenerate. Explicitly, if $\langle \ , \ \rangle$ denotes the standard inner product on $\mr^n$, we have $d_PE^*(F_X, G_Y) = \langle X, Y \rangle $ for all $X$, $Y \in \mr^n$. 
\end{lemma}

\begin{proof}
 Follows from \cite[Proposition 4]{harrach16}.
\end{proof}

\begin{rem}
In \cite{harrach16} the invariant form $E^*$ was used to construct for a large class of Lie groups a family of Poisson transforms $\Phi_k^{\mr} \colon \Omega^k(G/P, L) \to \Omega^k(G/K)$ for all $0 \le k \le n$ and any line bundle $L$ over $G/P$. In the case $G = \SO(n+1,1)_0$ it can be shown (c.f. \cite[section 3.2.]{harrach17}) that the images of these transforms are coclosed and eigenforms for the Laplace operator on $G/K$ and that they are the unique transforms up to real multiples which preserve the degree. Furthermore, if $L$ is the trivial bundle they satisfy the relation
\begin{align*}
 d \circ \Phi_k^{\mr} = (-1)^k (n-2k) \Phi_{k+1}^{\mr} \circ d
\end{align*}
for all $0 \le k \le n-1$.
\end{rem}


Next, a direct computation shows that the standard representation $\mv$ splits as an $M$-module into the direct sum $\mv = \mv_{-1} \oplus \mv_{0} \oplus \mv_{-1}$, where $\mv_{\pm 1} \cong \mr$ are copies of the trivial representation and $\mv_0 \cong \mr^n$ corresponds to the standard representation of $M \cong \SO(n)$. This decomposition is compatible with the $|1|$-grading of $\xg$ in the sense that $\xg_i \cdot \mv_j \subset \mv_{i+j}$ for all $i, j \in \{-1,0,1\}$, where we agree that $\mv_k = \{0\}$ for $|k| > 1$. Note that the above splitting of $\mv$ coincides with the eigendecomposition of the action of the grading element of $\xg$, and the indices are precisely the eigenvalues. 

As a last step we will identify the structure of the standard tractor bundles over $G/P$ and $G/K$. In order to do so, recall that for all $\lambda \in \mr$ we can define a $1$-dimensional $P$-representation $\mr[\lambda]$ via $p \cdot t = a^{-\lambda}t$ for all $p = man \in P$ and $t \in \mr$. Its associated line bundle $\ce[\lambda]$ is called the \emph{bundle of $\lambda$-densities}, and for any vector bundle $W \to G/P$ we define $W[\lambda] := W \otimes \ce[\lambda]$.

\begin{lemma}\label{lem_structure_standard_tractor_bundle}
 Let $G = \SO(n+1,1)_0$, $K$ its maximal compact subgroup, $P$ its parabolic subgroup and $\mv = \mr^{n+1,1}$ the standard representation of $G$. 
 \begin{enumerate}[(i)]
  \item The standard tractor bundle $V_K = G \times_K \mv$ over $G/K$ decomposes into
  the Whitney sum of $T(G/K)$ and a trivial line bundle.
  \item The standard tractor bundle $V_P = G \times_P \mv$ over $G/P$ carries a $G$-invariant filtration $V_P = V^{-1} \supset V^0 \supset V^1$ into homogeneous bundles whose associated graded vector bundle is given by
\begin{align*}
 \gr(V_P) = \ce[1] \oplus T(G/P)[1] \oplus \ce[-1].
\end{align*}
In particular, there is a canonical inclusion $\ce[-1] \hookrightarrow V_P$.
 \end{enumerate}
\end{lemma}
\begin{proof}

 \begin{enumerate}[(i)]
  \item The the maximal compact subgroup $K$ stabilizes a negative line as well as its orthogonal complement $\mv_T$, the latter being isomorphic to $\xg/\xk$ as a $K$-module. By the naturality of the associated bundle construction the claim follows.

\item The parabolic subgroup $P$ stabilizes a null line $\mv^1$ as well as its orthogonal complement $\mv^0 := (\mv^1)^{\perp}$. Thus, we obtain a $P$-invariant filtration $\mv = \mv^{-1} \supset \mv^0 \supset \mv^1$ of $\mv$, whose associated graded satisfies $\mv^{-1}/\mv^0 \cong \mr[1]$, $\mv^0/\mv^1 \cong (\xg/\xp)[1]$ and $\mv^1 \cong \mr[-1]$. Now the claim follows by the associated bundle construction. \qedhere
 \end{enumerate}
\end{proof}

\subsection{Construction of Poisson transforms }\label{sec_construction_tractor_valued_poisson}
We continue our discussion by explicitly constructing a family of Poisson kernels $\phi^{\mv} \in \Lambda^{\ell,n-k} (\xg/\xm)^* \otimes \End(\mv)$ whose associated Poisson transforms factor to the BGG-complex. In order to do so, note that the natural inclusion $\ce[-1] \hookrightarrow V_P$ from Lemma \ref{lem_structure_standard_tractor_bundle}(ii) induces an identification of $\Omega^k(G/P)$ with a subset of $\Omega^k(G/P, V_P[1])$. Furthermore, the splitting of $V_K$ in part (i) the same Lemma induces a decomposition 
\begin{align*}
 \Omega^\ell(G/K, V_K) = \Omega^\ell(G/K, T(G/K)) \oplus \Omega^\ell(G/K)
\end{align*}
for all $0 \le \ell \le n+1$ by tensoriality. Regarding the case of Poisson transforms between scalar valued differential forms we therefore focus on the construction of operators $\Phi_k^{\mv}$ between $V_P$-valued $k$-forms on $G/P$ and $T(G/K)$-valued $k$-forms on $G/K$ which factor to the BGG-complex and have coclosed images. If $\mv_T$ denotes the $K$-invariant subspace in $\mv$ corresponding to the tangent bundle as before, these transforms correspond by Theorem \ref{thm_Poisson_transform_differential_operators} to Poisson kernels $\phi_k^\mv \in \Lambda^{k,n-k}(\xg/\xm)^* \otimes L(\mv, \mv_T)$ which satisfy 
\begin{align*}
 \partial_P^*\phi_k^{\mv} &= 0, & \partial_P^*d_P\phi_k^{\mv} &= 0, & \delta_K\phi_k^{\mv} &= 0.
\end{align*}
In order to construct such a Poisson kernel we need to determine an explicit formula for the $P$-codifferential.
\begin{rem}
 For $k = n$ the first condition and for $k = 0$ the last two conditions on the Poisson kernel are trivially satisfied, implying that these two cases have to be dealt with seperately. Therefore, we will focus on the construction of $\phi_k^{\mv}$ for $1 \le k \le n-1$ in the sequel.
\end{rem}

\begin{prop}\label{prop_relations_P_codifferential}
 Let $\rho^{\mv^*}$ denote the tensor product of the induced $\xg$-action on $\mv^*$ with the identity on $\Lambda^*(\xg/\xm)^* \otimes \mv$. For $X \in \mr^n$ let $\xi_{X} \in \xg_1$ and $G_{X} \in \xk/\xm$ denote the corresponding vectors. Then the image of a Poisson kernel $\phi \in \Lambda^*(\xg/\xm)^* \otimes \End(\mv)$ under the $P$-codifferential is given by
  \begin{align}\label{eqn_P_codifferential}
 \partial_P^*\phi = -\frac{1}{2n} \sum_{j =1}^n \iota_{G_{e_j}} \rho^{\mv^*}_{\xi_{e_{j}}}\phi,
\end{align}
where $\{e_1, \dotsc, e_n\}$ be an orthonormal basis of $\mr^n$ with respect to the standard inner product on $\mr^n$. Moreover,
\begin{enumerate}[(i)]
 \item If $E \in \xg/\xm$ denotes the $M$-invariant vector induced by the grading element, then for all $\phi \in \Lambda^*(\xg/\xm)^* \otimes \End(\mv)$ we have
  \begin{align*}
 \partial_P^*\iota_E\phi &= - \iota_E\partial_P^* \phi, & \partial_P^*(E^* \wedge\phi) = - E^* \wedge \partial_P^*\phi.
\end{align*}
\item Let $\mw$ be any subspace of $\mv$. Then the image of $\phi \in \Lambda^*(\xg/\xm)^* \otimes L(\mv_j,\mw)$ under the $P$-codifferential has values in the subspace $L(\mv_{j-1}, \mw)$.
\item The subspace $\Lambda^{k, \ell}(\xg/\xm)^* \otimes L(\mv_{-1}, \mv)$ is contained in the image of the $P$-codifferential for all $0 \le \ell \le n-1$.
\end{enumerate}
\end{prop}

\begin{proof} 
Since the Lie algebra $\xg$ is $|1|$-graded, the formula for the Kostant codifferential form the proof of Proposition \ref{prop_Kostant_codifferential_self_adjoint} reduces to the alternation of the $\xp_+$-action on $\mv$. Explicitly, if we denote by $\eta_X \in \xg_{-1}$ the element corresponding to $X \in \mr^n$ we obtain that
\begin{align*}
\partial^*\psi = - \frac{1}{2n} \sum_{s =1}^n \rho_{\xi_{e_j}}^{\mv} \iota_{\eta_{e_j}}\psi
\end{align*}
for all $\psi \in \Lambda^{\ell} \xp_+ \otimes \mv$, where we used the explicit expression of the Killing form on $\xg$. Interpreting this formula in terms of the $M$-module $\xg/\xm$, the subspace $\xg_{-1} \cong \xg/\xp$ corresponds to $\xk/\xm$, which is generated by vectors $G_{e_j}$ for $j = 1, \dotsc, n$. Thus, on decomposable elements $\phi = \alpha \wedge \beta$ with $\alpha \in \Lambda^{\ell,0} (\xg/\xm)^* \otimes \mv$ and $\beta \in \Lambda^{0,n-k} (\xg/\xm)^* \otimes \mv^*$ we deduce that
\begin{align*}
\partial_P^* \phi = (-1)^{\ell+1} \frac{1}{2n} \sum_{j=1}^n \alpha \wedge \rho_{\xi_{e_j}} \iota_{G_{e_j}}\beta = -\frac{1}{2n} \sum_{j=1}^n \iota_{G_{e_j}}\rho^{\mv^*}_{\xi_{e_j}}\phi,
\end{align*}
where we used that $\alpha$ is of bidegree $(\ell,0)$ and thus trivial upon insertion of the vectors $G_{e_j}$ for all $j = 1, \dotsc, n$. Since $\partial_P^*$ is linear this formula holds for all elements in $\Lambda^*(\xg/\xm)^* \otimes \End(\mv)$.

 \begin{enumerate}[(i)] 
  \item Follows immediately from (\ref{eqn_P_codifferential}) by using the antiderivation property of the interior product.
  
  \item For all $\xi \in \xp_+ = \xg_1$ the values of the $M$-invariant elements $\rho^{\mv^*}_{\xi}\phi$ lie in the space $L(\mv_{j-1}, \mw)$ by definition of the dual action.

\item It suffices to construct a preimage of $\phi \in \Lambda^{0,\ell}(\xg/\xm)^* \otimes L(\mv_{-1}, \mv)$ under $\partial^*_P$, which has to be of bidegree $(0,\ell+1)$ and to have values in $L(\mv_0, \mv)$ by part (ii). Explicitly, if we denote by $F_X \in \xp/\xm$ and $\eta_X \in \xg_{-1}$ the elements corresponding to $X \in \mr^n$, then a direct computation shows that
\begin{align*}
 \tilde{\phi} := \frac{2n}{n-1} \sum_{i=1}^n (\iota_{F_{e_i}}d_PE^*) \wedge \rho^{\mv^*}_{\eta_{e_i}}\phi
 \end{align*}
satisfies $\partial_P^*\tilde{\phi} = \phi$. \qedhere
 \end{enumerate}
\end{proof}

In order to get an appropriate ansatz for the Poisson kernels $\phi_k^{\mv}$ consider the $M$-equivariant map $S \in \End(\mv)$ which induced the matrix representation of $G$ in section \ref{sec_geometric_preliminaries}. This map satisfies $\theta(X)\cdot S(v) = S(X \cdot v)$ for all $X \in \xg$ and $v \in \mv$, which implies that it is contained in the kernels of the partial derivatives $d_K^\theta$ and $d_P$. Furthermore, for $i \in \{1,-1\}$ let $f_i \colon \mv_i \to \mv_1$ denote the $M$-equivariant maps induced by the identity. Then we define for all $k \in \mr$ the $M$-invariant element $\sigma_k \in \Lambda^{1,0}(\xg/\xm)^* \otimes L(\mv, \mv_T)$ via
\begin{align*}
 \sigma_k := (k+1)E^* \otimes S - (k+1)d_K^\theta f_1 + (n+2)d_K^\theta f_{-1}.
\end{align*}
By construction, these elements are contained in the kernel of $d_K^{\theta}$, satisfy $\sigma_k(\xi)|_{\mv_1} = 0$ and $\sigma_{-1}(\xi)|_{\mv_0} = 0$ for all $\xi \in \xp/\xm$ as well as the relation
\begin{align}\label{eqn_sigma}
 k\sigma_k + \sigma_{-1} = (k+1)\sigma_{k-1}
\end{align}
for all $k \in \mr$. Furthermore, using the formula for the $P$-codifferential from Proposition \ref{prop_relations_P_codifferential} a direct computation shows that $(d_PE^*)^{n-k-1} \wedge d_P\sigma_k$ is contained in the kernel of $\partial_P^*$ for all $1 \le k \le n-1$.

Subsequently, we consider the $M$-invariant element
\begin{align*}
 \phi^{(1)}_k := \ast_K(E^* \wedge (d_PE^*)^{n-k-1} \wedge d_P\sigma_k) \in \Lambda^{k,n-k} (\xg/\xm)^* \otimes L(\mv, \mv_T)
\end{align*}
for all $1 \le k \le n-1$, which by definition of the $K$-codifferential and Lemma \ref{lem_E_invariant} is $K$-coclosed. Furthermore, the $P$-codifferential commutes with the $K$-Hodge star up to a sign as well as with the wedge product with $E^*$ due to Proposition \ref{prop_relations_P_codifferential}(i), implying that $\partial_P^*\phi^{(1)}_k = 0$. However, the $P$-differential of this form is not in the kernel of $\partial^*_P$.  Therefore, we add an additional error term, which in view of relation (\ref{eqn_sigma}) we choose as a multiple of
\begin{align*}
 \phi^{(2)}_k := \ast_K((d_PE^*)^{n-k} \wedge \sigma_{-1}) \in \Lambda^{k,n-k} (\xg/\xm)^* \otimes L(\mv_{-1}, \mv_T).
\end{align*}
This is contained in the kernel of $\partial^*_P$ due to Proposition \ref{prop_relations_P_codifferential}(ii) and also $K$-coclosed by the same arguments as before.

Therefore, for all $1 \le k \le n-1$ the Poisson kernel $\phi_k^{\mv} := k\phi^{(1)}_k + \phi^{(2)}_k$ is contained in the kernels of $\delta_K$ and $\partial_P^*$. Furthermore, using that $d_P\ast_K = (-1)^{n+1}\ast_Kd_P$ from Theorem \ref{thm_Poisson_transform_differential_operators} as well as relation (\ref{eqn_sigma}) its $P$-derivative satisfies
\begin{align}\label{eqn_P_derivative_Poisson_kernel}
 d_P\phi_k^{\mv} = (-1)^{n+1} (k+1)\ast_K((d_PE^*)^{n-k} \wedge d_P\sigma_{k-1}),
\end{align}
which is also contained in the kernel of $\partial_P^*$ by a direct computation. Thus, by Corollary \ref{cor_bgg_poisson_kernel} it follows that the Poisson transforms $\Phi_k^{\mv}$ induced by $\phi_k^{\mv}$ factor to $G$-equivariant operators
\begin{align*}
 \underline{\Phi}_k^{\mv} \colon \Gamma(\ch_k(G/P,V_P)) \to \Omega^k(G/K, T(G/K))
\end{align*}
for all $1 \le k \le n-1$.

\begin{rem}
The definition of $\phi_k^{\mv}$ initially also makes sense for $k = 0$. However, since $\iota_E\sigma_{-1} = 0$ it follows that the resulting Poisson kernel is trivial.
\end{rem}

\begin{thm}
 Let $G = \SO(n+1,1)_0$, $K \subset G$ its maximal compact subgroup and $P \subset G$ its minimal parabolic subgroup. Let $\mv = \mr^{n+1,1}$ be the standard representation of $G$ and $V_P = G \times_P \mv$ its associated vector bundle over $G/P$. For all $1 \le k \le n-1$ let
 \begin{align*}
  \underline{\Phi}_k^{\mv} \colon \Gamma(\ch_k(G/P, V_P)) \to \Omega^k(G/K, T(G/K))
 \end{align*}
the $G$-equivariant map induced by the Poisson kernels $\phi_k^{\mv}$. Then the images of $\underline{\Phi}_k^{\mv}$ are coclosed and harmonic differential forms on $G/K$. 
\end{thm}

\begin{proof}
Since the Poisson transform induced by $\phi_k^{\mv}$ factor to the BGG-complex and have coclosed image, the claim follows from Theorems \ref{thm_harmonic_bgg} and Theorem \ref{thm_Poisson_transform_differential_operators}.
\end{proof}

\begin{rem}
 With more effort it can be shown that for $1 \le k \le n-1$ the operator $\Phi_k^{\mv}$ is the unique Poisson transform $\Omega^k(G/P,V_P) \to \Omega^k(G/K,V_K)$ up to multiples which factors to the BGG-complex.
\end{rem}

\subsection{The Poisson transforms on weighted differential forms}
In the next step we extend the domain of the Poisson transforms constructed in the previous section to differential forms on $G/P$ with values in weighted standard tractor bundles $V_P[\lambda]$ for any $\lambda \in \mr$. We show that the image of their induced $G$-equivariant operators on sections of the weighted homology bundles $\ch_k(G/P,V_P)[\lambda]$ consists again of coclosed differential forms which are eigenforms of the covariant Laplace operator on $G/K$.


Explicitly, fix $\lambda \in \mr$ and consider the pullback of $\alpha \in \Omega^k(G/P, V_P[\lambda])$ along the projection $\pi_P \colon G/M \to G/P$. Since $M$ is compact, the pullback of a $\lambda$-density is a smooth function on $G/M$, so $\pi_P^*\alpha$ is a differential form on $G/M$ with values in $G \times_M \mv$. Therefore, the same formula as in the unweighted case defines a $G$-equivariant operator $\Phi_{k,\lambda}^{\mv} \colon \Omega^k(G/P, V_P[\lambda]) \to \Omega^k(G/K, T(G/K))$ induced by the Poisson kernel $\phi_k^{\mv}$, which we again call a Poisson transform. 

Via the proof of Theorem \ref{thm_Poisson_transform_differential_operators} we can compute the image of $\alpha \in \Omega^k(G/P, V_P[\lambda])$ under $\Delta_K \circ \Phi_{k,\lambda}^{\mv}$ as the fibre integral of $\Delta_K(\phi_k^{\mv} \wedge \pi_P^*\alpha)$ over $G/P$, and similarly for $\delta_K \circ \Phi_{k, \lambda}^{\mv}$. However, since the bundle $V_P[\lambda]$ does not carry a $G$-invariant connection, the results from Theorem \ref{thm_Poisson_transform_differential_operators}(ii) on the composition of the covariant operators on $G/K$ do not apply. Thus, we first need to prove the following Lemma.

\begin{lemma}\label{lem_K_derivative_density}
Let $\sigma$ be a $\lambda$-density on $G/P$. Then the $K$-derivative of the pullback $\pi_P^*\sigma \in C^{\infty}(G/M)$ is given by
\begin{align*}
 d_K\pi_P^*\sigma = \lambda \pi_P^*\sigma \cdot E^*.
\end{align*}
\end{lemma}
\begin{proof}
 If $f \colon G \to \mr[\lambda]$ is the $P$-equivariant map corresponding to $\sigma$, then the pullback $\pi_P^*\sigma$ corresponds to the same map, viewed as an $M$-equivariant function. Thus, its $K$-derivative $d_K\pi_P^*\sigma \in \Omega^{1,0}(G/M)$ corresponds to the equivariant map $d_Kf \colon G \to (\xp/\xm)^* \otimes \mr[\lambda]$ given by $(d_Kf)(X + \xm)(g) = (L_X f)(g)$ for all $g \in G$, where $L_X \in \mathfrak{X}(G)$ denotes the left invariant vector field determined by $X \in \xg$. We will compute the derivative explicitly by using the Langlands decomposition $\xp = \xm \oplus \xa \oplus \xn$. First, the subalgebras $\xn = \xg_1$ and $\xm$ act trivially on $\mr[\lambda]$, implying that $d_Kf$ is trivial for $X \in \xm \oplus \xn$. On the other hand, the subalgebra $\xa$ is generated by the grading element $\tilde{E}$, which projects to the $M$-invariant vector $E \in \xg/\xm$. By $P$-equivariance of $f$ we obtain
 \begin{align*}
  (d_Kf)(g)(E) = (L_{\tilde{E}}f)(g) = \lambda \tilde{E} \cdot f(g),
 \end{align*}
and since the grading element acts by the identity on $\mr[\lambda]$ the claim follows.
\end{proof}

Next, we define several $M$-equivariant linear operators related to the invariant vector $E \in \xg/\xm$ which appear in the formula for $\Delta_K(\phi_k^{\mv} \wedge \pi_P^*\alpha)$ and compute the images of Poisson kernels under these operator. First of all, consider the $\xg$-actions $\rho$ and $\rho^\theta$ on $f \in \End(\mv)$ by
\begin{align*}
 \left(\rho_X f\right)(v) &:= X \cdot f(v) - f(X \cdot v), & \left(\rho^{\theta}_X f\right)(v) &:= \theta(X) \cdot f(v) - f(X \cdot v)
\end{align*}
for all $v \in \mv$ and $X \in \xg$, and denote their trivial extensions to $\phi \in \Lambda^k(\xg/\xm)^* \otimes \End(\mv)$ by the same symbols. Via this notation, the $M$-equivariant map corresponding to the derivative is for all $\xi_j = X_j + \xm \in \xg/\xm$ given by the formula
\begin{equation}\label{eqn_derivative}
\begin{aligned} d\phi(\xi_0, \dotsc, \xi_k) =\ &\sum_{i=0}^k (-1)^i \rho_{X_i}(\phi)(\xi_0, \dotsc, \hat{\imath}, \dotsc, \xi_k) \\
 &+ \sum_{i<j} (-1)^{i+j} \phi([X_i, X_j] + \xm, \xi_0, \dotsc, \hat{\imath}, \dotsc, \hat{\jmath}, \dotsc, \xi_k),
 \end{aligned}
\end{equation}
where the hat denotes omission. Next, we define for all $\xi \in \xp/\xm$ the maps $\cl_\xi := \iota_{\xi}d_K + d_K\iota_{\xi}$, and similarly for $\cl_{\xi}^\theta$, as well as $\cl_{\xi}^*:= \xi^{\flat} \wedge \delta_K + \delta_K \xi^{\flat} \wedge$, which are the infinitesimal versions of the Lie derivatives and the adjoint with respect to the $K$-Hodge star operator. In addition, for each subspace $\mw \subset \mv$ we denote by $\phi|_{\mw}$ the restriction of the values of $\phi$ to $\mw$, i.e.
\begin{align*}
 (\phi|_{\mw})(\xi_1, \dotsc, \xi_k) := \phi(\xi_1, \dotsc, \xi_k)|_{\mw}
\end{align*}
for all $\xi_1, \dotsc, \xi_k \in \xg/\xm$. 

\begin{lemma}\label{lem_lie_derivative}
Let $\phi$ be any Poisson kernel of bidegree $(\ell,n-k)$ with values in $\End(\mv)$. Then
\begin{enumerate}[(i)]
\item $\cl_E\phi =  \rho_E(\phi) + (n-k-\ell) \phi + E^* \wedge \iota_E\phi$ and similarly for $\cl_E^{\theta}$ by exchanging $\rho_E$ with $\rho_E^{\theta}$.
\item $\cl_E^*\phi = -\rho_E^{\theta}(\phi) - (\ell - k - 1)\phi - \iota_E(E^* \wedge \phi)$.
\item The grading element $\tilde{E}$ satisfies $\rho_{\tilde{E}}(\phi) + \rho_{\tilde{E}}^\theta(\phi) = 2\left(\left.\phi\right|_{\mv_{-1}} - \left.\phi\right|_{\mv_{1}}\right)$.
\end{enumerate}
\end{lemma}
\begin{proof}
\begin{enumerate}[(i)]
 \item Follows by a direct computation using (\ref{eqn_derivative}) and the fact that the grading element acts on $\xg_i$ by multiplication with $i$.

\item Since we can write $\cl_E^*\phi = -\ast_K^{-1} \cl_E^{\theta}(\ast_K \phi)$ the claim follows from (i) by applying the identity $E^* \wedge \iota_E(\ast_K\phi) = \ast_K \iota_E(E^* \wedge \phi)$.

\item The grading element is in the $(-1)$-eigenspace of the Cartan involution $\theta$, which implies for $f \in \End(\mv)$ that
\begin{align*}
(\rho_{\tilde{E}} + \rho^\theta_{\tilde{E}})(f)(v) = -2f(\tilde{E} \cdot v)
\end{align*}
for all $v \in \mv$. Furthermore, $\tilde{E}$ acts on $\mv_j$ by multiplication with $j$, so the right hand side equals $2(f|_{\mv_{-1}} - f|_{\mv_1})(v)$. Since the $\xg$-action on the  space of Poisson kernels is the trivial extension of the action on $\End(\mv)$, the claim follows.
\qedhere
\end{enumerate}
\end{proof}

Using these formulae we can determine the images of a weighted standard tractor bundle valued differential form on $G/P$ under the compositions of the Poisson transform with the covariant Laplace operator and the covariant codifferential on $G/K$, respectively. 

\begin{prop}\label{prop_Delta_K_density}
 For $1 \le k \le n-1$ let $\phi_k^{\mv}$ be the Poisson kernel constructed in section \ref{sec_construction_tractor_valued_poisson} and denote the corresponding $G$-invariant differential form by the same symbol. Then for all $\alpha \in \Omega^k(G/P, V_P[\lambda])$ we have 
 \begin{align*}
 \delta_K(\phi_k^{\mv} \wedge \pi_P^*\alpha) &=  - \lambda (\iota_E\phi_k^{\mv}) \wedge \pi_P^*\alpha\\
  \Delta_K(\phi_k^\mv \wedge \pi_P^*\alpha) &= -\lambda(n-2k + \lambda)\phi_k^{\mv}\wedge \pi_P^*\alpha - 2\lambda \left(\left.\phi_k^{\mv}\right|_{\mv_{-1}}\right) \wedge \pi_P^*\alpha.
 \end{align*}
\end{prop}

\begin{proof}
For any Poisson kernel $\phi \in \Omega^{\ell,n-k}(G/M, G \times_M \End(\mv))$ we define the $(k,n)$-form $\phi_{\lambda} := \phi \wedge \pi_P^*\alpha$ on $G/M$. Combining the formula for the $K$-derivative of the pullback of a density from Lemma \ref{lem_K_derivative_density} and the definition of the $K$-codifferential we directly compute that
 \begin{align*}
  d_K\phi_{\lambda} &= \left(d_K\phi + \lambda E^* \wedge \phi\right) \wedge \pi_P^*\alpha,  & \delta_K\phi_{\lambda} &=\left(\delta_K\phi - \lambda \iota_E\phi\right) \wedge \pi_P^*\alpha.
 \end{align*}
Putting these two formulae together and applying Lemma \ref{lem_lie_derivative} we can write the image of $\phi_{\lambda}$ under the $K$-Laplace as
\begin{align*}
 \Delta_K\phi_{\lambda} &= \left(\Delta_K - \lambda \cl_E + \lambda \cl_E^* - \lambda^2\right)\phi \wedge \pi_P^*\alpha \\
 &= \left(\Delta_K\phi + 2 \lambda \phi|_{\mv_1} - 2 \lambda \phi|_{\mv_{-1}} - \lambda(n-2k + \lambda) \phi\right) \wedge \pi_P^*\alpha.
\end{align*}
In particular, for $\phi = \phi_k^{\mv}$ we have $\Delta_K\phi = 0$, $\delta_k\phi = 0$ and $\phi|_{\mv_1} = 0$. 
\end{proof}

In particular, combining Proposition \ref{prop_Delta_K_density} with Theorem \ref{thm_Poisson_transform_differential_operators} we see that for $\alpha \in \Omega^k(G/P, V_P[\lambda])$ the image $\Phi_k^{\mv}(\alpha)$ is in general not an eigenform of the twisted Laplace operator $\Delta^{V_K}$ on $G/K$. However, if we factor the Poisson transforms to the corresponding homology bundles, we have the following Theorem. 

\begin{thm}\label{thm_twisted_BGG}
 Let $G = \SO(n+1,1)_0$, $K \cong \SO(n+1)$ its maximal compact subgroup and $P$ its minimal parabolic subgroup. Let $\mv$ be the standard representation of $G$ and $V_P$ the standard tractor bundles over $G/P$. For $\lambda \in \mr$ let 
 \begin{align*}
 \underline{\Phi}_{k, \lambda}^{\mv} \colon \Gamma(\ch_k(G/P, V_P[\lambda])) \to \Omega^k(G/K, T(G/K))
\end{align*}
be the $G$-equivariant linear operators induced by the Poisson kernels $\phi_k^{\mv}$. Then the image of $\underline{\Phi}_{k, \lambda}^{\mv}$ consists of coclosed eigenforms of the twisted Laplace Beltrami operator with eigenvalue $-\lambda(n-2k+\lambda)$.
\end{thm}

\begin{proof} For $\sigma \in \ch_k(G/P, V_P[\lambda])$ choose any representative $\alpha \in \Omega^k(G/P, V_P[\lambda])$ contained in the kernel of the Kostant codifferential. By Theorem \ref{thm_Poisson_transform_differential_operators}, the image of the $V_K$-valued differential form $\underline{\Phi}_k^{\mv}(\sigma)$ on $G/K$ under the covariant codifferential $\delta^{V_K}$ and the covariant Laplace operator $\Delta^{V_K}$ is given by the fibre integral of $\delta_K(\phi_k^{\mv} \wedge \pi_P^*\alpha)$ and $\Delta_K(\phi_k^{\mv} \wedge \pi_P^*\alpha)$, respectively. Using the results from Proposition \ref{prop_Delta_K_density} we need to show that the Poisson transforms associated to $\iota_E\phi_k^{\mv}$ and $\left. \phi_k^{\mv}\right|_{\mv_{-1}}$ are trivial on the kernel of the Kostant codifferential. By selfadjointness of the $P$-codifferential it suffices to prove that these two Poisson kernels are in the image $\partial_P^*$. However, by construction both these $M$-invariant elements are contained in $\Lambda^*(\xg/\xm)^* \otimes L(\mv_{-1}, \mv)$ and thus in the image of $\partial_P^*$ due to Proposition \ref{prop_relations_P_codifferential}(iii).
\end{proof}

\subsection{Compatibility of PTs with BGG-operators}
We conclude our discussion by proving that the Poisson transforms $\Phi_k^{\mv}$ also satisfy a commutation relation with the covariant exterior derivatives. We will do this by showing that the partial derivatives of the Poisson kernels coincide up to a multiple. 

First, we have to introduce some notation. By a \emph{$k$-vector} ${\bf X}$ of $\mr^n$ we mean a $k$-fold wedge product $X_1 \wedge \dotso \wedge X_k$ of vectors in $\mr^n$. We denote by $F_{\bf X}$ the wedge product $F_{X_1} \wedge \dotso \wedge F_{X_k}$, and similarly for $G_{\bf X}$. Finally, for $1 \le i \le k$ we write $F_{\bf X}^i$ for the wedge product obtained from $F_{\bf X}$ by omitting the $i$-th factor.

\begin{prop}\label{prop_density_K_derivative}
For all $k = 1, \dotsc, n-2$ the Poisson transforms $\Phi_k^{\mv}$ satisfy the relation
\begin{align*}
 (k+2) d^{V_K} \circ \Phi_k^{\mv} &= (-1)^{n-k} k (n-2k)  \Phi_{k+1}^{\mv} \circ d^{V_P}.
\end{align*}
In particular, if $n$ is even, then the image of $\Phi_{\frac{n}{2}}$ is contained in the space of closed and coclosed $V_K$-valued differential forms on $G/K$.
\end{prop}

\begin{proof}
By Theorem \ref{thm_Poisson_transform_differential_operators} we need to determine the $K$-derivative of $\phi_k^{\mv}$ and relate it the $P$-derivative of $\phi_{k+1}^{\mv}$, c.f. equation (\ref{eqn_P_derivative_Poisson_kernel}) in section \ref{sec_construction_tractor_valued_poisson}. In order to do so, for all $X$, $Y \in \mr^n$ we denote by $\xi_{X}$ and $\eta_Y$ the corresponding elements in $\xp_+$ and $\xk$ and by $F_X \in \xp/\xm$ and $G_Y \in \xk/\xm$ their projection to $\xg/\xm$, respectively. First, we will determine the Poisson kernels $E^* \wedge \iota_Ed_K\phi_k^{\mv}$ and $\iota_E(E^* \wedge d_K\phi_k^{\mv})$ using formula (\ref{eqn_derivative}) for the infinitesimal version of the derivative and then combine them to obtain the claimed result.

\begin{enumerate}[(i)]
\item Let $\textbf{X}$ be a $k$-vector and $\textbf{Y}$ be an $(n-k)$-vector on $\mr^n$. Applying formula (\ref{eqn_derivative}) for the $K$-derivative we can split
\begin{align*}
 E^* \wedge \iota_Ed_K\phi_k^{\mv}(E, F_\textbf{X}, G_\textbf{Y}) = d_K\phi_k^{\mv}(E, F_\textbf{X}, G_\textbf{Y}) = (*)
\end{align*}
into $(*) = (I) + (II)$, where $(I)$ is the sum of all expressions involving Lie brackets of representatives of the vectors in $\xg/\xm$ and $(II)$ is the sum of all expressions involving the action of $\xg$ on $\End(\mv)$.

For the first summand the bracket relations on $\xg$ imply that the only pairings which contribute nontrivially to the $K$-derivative are those between the grading element and the representatives $\xi_{X_i}$ and $\eta_{Y_j}$, which equal $\xi_{X_i}$ and $2\xi_{Y_j} - \eta_{Y_j}$, respectively. Thus, a direct computation shows that
\begin{align*}
 (I) = (n-2k)\left(E^* \wedge \phi_k^{\mv}\right)(E, F_\textbf{X}, G_\textbf{Y}).
\end{align*}

For the other summand of $(*)$ we have to alternate over the actions of the representatives of the vectors in $\xp/\xm$, i.e. we have to determine
\begin{align*}
 (II) = \rho_E\phi_k^{\mv}(F_\textbf{X}, G_\textbf{Y}) + \sum_{i=1}^k (-1)^{i+1} \left(\rho_{\xi_{X_i}}\phi_k^{\mv}\right)(E, F_\textbf{X}^i, G_\textbf{Y}).
\end{align*}
 Using the explicit formulae for the $\xg$-action on $\mv$ and the exterior calculus on $(\xg/\xm)^*$ a direct computation shows that 
 \begin{align*}
  k \iota_E\rho_{\xi_X}\phi_k^{\mv} = \iota_E(E^* \wedge \iota_{F_X}\rho_E\phi_k^{\mv})
 \end{align*}
for all $X \in \mr^n$. Inserting this into the second sum of $(II)$ and moving the vectors $F_{X_i}$ to the right place we deduce that $(II) = 0$. 

\item By definition, the expression $\iota_E\left(E^* \wedge d_K\phi_k^{\mv}\right)$ is only nontrivial upon insertion of $F_\textbf{X}$ and $G_\textbf{Y}$ for any $(k+1)$-vector $\textbf{X}$ and $(n-k)$-vector $\textbf{Y}$. Therefore, we have to compute the expression
\begin{align*}
 \iota_E\left(E^* \wedge d_K\phi_k^{\mv}\right)(F_\textbf{X}, G_\textbf{Y}) = d_K\phi_k^{\mv}(F_\textbf{X}, G_\textbf{Y}) = (**),
\end{align*}
for which we again use formula (\ref{eqn_derivative}). Note that the Lie bracket between two elements in $\xp_+$ is trivial, and that $[\xi_X, \eta_Y] = -\langle X, Y \rangle E$ is not contained in $\xk$. Therefore, the bracket part of $(**)$ is trivial, obtaining
\begin{align*}
 (**) &= \sum_{i=1}^{k+1} (-1)^{i+1}\rho_{\xi_{X_i}}\phi_k^{\mv}\left(F_\textbf{X}^i, G_\textbf{Y}\right).
\end{align*}
Using the definition of the Poisson kernel $\phi_k^{\mv}$ and the exterior calculus on $\Lambda^* (\xg/\xm)^*$ it can be shown via a lengthy computation (c.f. \cite[Proposition 3.4.3(ii)]{harrach17}) that
\begin{align*}
 (**) &= (-1)^{n-k} k(n-2k) \ast_K(E^* \wedge (d_PE^*)^{n-k-1} \wedge \iota_Ed_P\sigma_k)(F_{\bf X}, G_{\bf Y}).
\end{align*}
Finally, we can distribute the interior product with $E$ over the wedge product and apply the formula $\ast_K(\iota_E\omega) = (-1)^{\ell+1} E^* \wedge \ast_K\omega$, which holds for all Poisson kernels $\omega$ of bidegree $(\ell,n-k)$, resulting in
\begin{align*}
 \iota_E\left(E^* \wedge d_K\phi_k^{\mv}\right) =\ &(-1)^{n-k} k(n-2k) \ast_K((d_PE^*)^{n-k-1} \wedge d_P\sigma_k) \\
 &- (n-2k) E^* \wedge \phi_k^{\mv}.
\end{align*}
\end{enumerate}
Combining the formulae from parts (i) and (ii) we obtain
\begin{align*}
 (k+2) d_K\phi_k^{\mv} &= (k+2)\left(E^* \wedge \iota_Ed_K\phi_k^{\mv} + \iota_E\left(E^* \wedge d_K\phi_k^{\mv}\right)\right) \\
 &= (-1)^{k+1} k(n-2k)\left((-1)^{n+1}(k+2) \ast_K((d_PE^*)^{n-k-1} \wedge d_P\sigma_k)\right),
\end{align*}
and the expression in the brackets coincides with the $P$-derivative of $\phi_{k+1}^{\mv}$, c.f. equation (\ref{eqn_P_derivative_Poisson_kernel}) in section \ref{sec_construction_tractor_valued_poisson}. Finally, for the global result we apply Theorem \ref{thm_Poisson_transform_differential_operators}.
\end{proof}

The last result also induces compatibility results for the BGG-operators on $G/P$ and the induced $G$-equivariant maps $\underline{\Phi}_k^{\mv}$ induced by the Poisson transforms. Indeed, recall that we constructed the Poisson transforms $\Phi_k$ to satisfy $\Phi_k \circ d^{V_P} \circ \partial^*$, which in turn implies that for any $\sigma \in \Gamma(\ch_k(G/P,V_P))$ and any representative $\alpha \in \Gamma(\ker(\partial^*))$ of $\sigma$ we have $\underline{\Phi}_k^{\mv}(D^{V_P}\sigma) = \Phi_k^{\mv}(d^{V_P}\alpha)$, c.f. section \ref{sec_equivariant_operators}. Thus, from Proposition \ref{prop_density_K_derivative} we immediately obtain

\begin{thm}\label{thm_relation_BGG_derivatives}
 Let $G = \SO(n+1,1)_0$, $K \cong \SO(n+1)$ its maximal compact subgroup and $P$ its minimal parabolic subgroup. Let $\mv = \mr^{n+1,1}$ be the standard representation of $G$ and denote by $V_K$ and $V_P$ the corresponding standard tractor bundles over $G/K$ and $G/P$, respectively. Then the $G$-equivariant maps
 \begin{align*}
  \underline{\Phi}_{k}^{\mv} \colon \Gamma(\ch_k(G/P, V_P)) \to \Omega^k(G/K, T(G/K))
 \end{align*}
induced by the Poisson kernels $\phi_k^{\mv}$ satisfy
 \begin{align*}
  (k+2) d^{V_K} \circ \underline{\Phi}_k^{\mv} = (-1)^{n-k} k(n-2k) \underline{\Phi}_{k+1}^{\mv} \circ D_{k},
 \end{align*}
 for all $k = 1, \dotsc, n-2$, where $D_k$ denotes the $k$-th BGG-operator.
\end{thm}

\address
\end{document}